\documentclass{amsproc}
\usepackage{amssymb}
\usepackage{tikz}
\usepackage{pgfplots}
\usepackage{multirow}

\usepackage{graphicx}
\usepackage{amsmath}
\usepackage{subfigure}
\usepackage{calc}
\usepackage{float}
\usepackage{wrapfig}
\usepackage{lipsum}
\usepackage{pstricks}

\newtheorem{theorem}{Theorem}[section]
\newtheorem{lemma}{Lemma}[section]
\newtheorem{proposition}{Proposition}[section]
\newtheorem{remark}{Remark}[section]

\newtheorem{definition}{Definition}[section]

\numberwithin{equation}{section}

\newcommand{\dotcup}{\ensuremath{\mathaccent\cdot\cup}}

\newcommand\blfootnote[1]{
  \begingroup
  \renewcommand\thefootnote{}\footnote{#1}
  \addtocounter{footnote}{-1}%
  \endgroup
}

\definecolor{gray9}{rgb}{0.9,0.9,0.9}

\def\Z{\mathbb{Z}}
\def\R{\mathbb{R}}

\def\N{\mathbb{N}}
\def\P{\mathbb{P}}

\def\n{{\bf{n}}}

\def\0{{\bf{0}}}

\def\vx{{\vec{x}}}

\def\vz{{\vec{z}}}
\def\vu{{\vec{u}}}
\def\vv{{\vec{v}}}
\def\vt{{\vec{t}}}
\def\vs{{\vec{s}}}

\def\vk{{\vec{k}}}
\def\vm{{\vec{m}}}
\def\vn{{\vec{n}}}
\def\va{{\vec{a}}}
\def\vb{{\vec{b}}}

\def\vl{\vec{\ell}}
\def\ve{{\vec{e}}}
\def\vf{{\vec{f}}}

\def\thematrix{ expansive dyadic integral }

\def\dotcup{\dot\cup}

\begin{document}

\allowdisplaybreaks

\begin{center}
\Large\textbf{Isomorphism in Wavelets}
\end{center}
\bigskip
\begin{center}
      \large\textit{Xingde Dai and Wei Huang}
\textsf{}\end{center}
      \address{The University of North Carolina at Charlotte; Wells Fargo Bank}
      \email{xdai@uncc.edu}

\allowdisplaybreaks
Two scaling functions $\varphi_A$ and $\varphi_B$ for Parseval frame wavelets are algebraically isomorphic, $\varphi_A \simeq \varphi_B$,
if they have matching solutions to their (reduced) isomorphic systems of equations.
Let $A$ and $B$ be $d\times d$ and $s\times s$  \thematrix matrices with $d, s\geq 1$  respectively
and let $\varphi_A$ be a scaling function associated with matrix $A$ and generated by a finite solution.
There always exists a scaling function $\varphi_B$ associated with matrix $B$ such that
\begin{equation*}
    \varphi_B \simeq \varphi_A.
\end{equation*}
An example shows that the assumption on the finiteness of the solutions can not be removed.


\section{Introduction}\label{intro}\label{ss:introduction}
For a vector $\vec{\ell}\in\R^d$,
the \textit{translation operator} $T_{\vec{\ell}}$ is defined as
\begin{eqnarray*}
    (T_{\vec{\ell}} f)(\vec{t}) &\equiv& f(\vec{t}-\vec{\ell}), ~ \forall f\in L^2(\R^d), \ \forall \vt\in\R^d\ .
\end{eqnarray*}
Let $A$ be a $d \times d$ integral matrix with eighenvalues $\beta_1, \cdots, \beta_d$. $A$ is called \textit{expansive} if $\min \{|\beta_1|, \cdots, |\beta_d|\} >1$.  $A$ is called \textit{dyadic}   if $|\det(A)|=2$.
We define the operator $D_{A}$ as
\begin{eqnarray*}
    (D_A f)(\vec{t}) &\equiv& (\sqrt{2}) f(A\vec{t}),~\forall f \in L^2(\R^d), \ \forall \vt\in\R^d\ .
\end{eqnarray*}
The operators $T_{\vec{\ell}}$ and $D_A$ are unitary operators on $L^2(\R^d)$.

\begin{definition}\label{def:defpsi0}
Let $A$ be an expansive dyadic integral matrix.
A function $\psi_A\in L^2(\R^d)$ is called a Parseval frame wavelet associated with $A$, if the set
\begin{equation*}
    \{ D_A ^n T_{\vec{\ell}}~\psi_A \mid  n\in\Z,\vec{\ell}\in\Z^d\}
\end{equation*}
forms a normalized tight frame for $L^2(\R^d)$.
That is
\begin{equation*}
    \|f\|^2 = \sum_{n\in\Z , \vec{\ell} \in \Z^d} |\langle f,D_A ^n T_{\vec{\ell}}~\psi_A \rangle |^2, \ \forall f\in L^2(\R^d).
\end{equation*}
If the set is also orthogonal, then $\psi_A$ is an orthonormal wavelet for $L^2(\R^d)$ associated with $A$.
\end{definition}

Consider the following system of equations \eqref{eq:lawton} associated with an expansive dyadic integral matrix $A$:
\begin{equation}\label{eq:lawton}
\left\{\begin{array}{l}
\sum_{\vec{n}\in\Z^d}h_{\vec{n}}\overline{h_{\vec{n}+\vec{k}}}=\delta_{\vec{0} \vec{k}},~ \vec{k}\in A\Z^d,\\
\sum_{\vec{n}\in\Z^d}h_{\vec{n}}=\sqrt{2}.
\end{array}\right.
\end{equation}
Let $\mathcal{S}=\{s_{\vn}\mid \vec{n}\in\Z^d\}$ be a solution.
The set $\Lambda = \{\vn \in \Z^d \mid s_\vn \neq 0\}$ is called the support of $\mathcal{S}$.
We will say the solution is finite if its support is a finite set.
Define the operator $\Psi$ on $L^2(\R^d)$ as
\begin{equation*}
  \Psi \equiv \sum_{\vec{n}\in\Z^d} s_{\vec{n}} D_AT_{\vec{n}}.
\end{equation*}
Bownik \cite{B} and Lawton \cite{lawton} proved that
the iterated sequence
\begin{equation}\label{eq:sequence}
\{\Psi^k \chi_{_{[0,1)^d}} \mid k\in\N\}
\end{equation}
converges in the $L^2(\R^d)$-norm.
The limit is the scaling function $\varphi_A$ associated with matrix $A$.
It induces a Parseval frame wavelet $\psi_A$ associated with matrix $A$.
The scaling function $\varphi_A$ satisfies the following two-scale relation \eqref{eq:twoeq1}:
\begin{eqnarray}\label{eq:twoeq1}
   \varphi_{A} &=& \sum_{\vec{n}\in\Z^d} s_{\vec{n}} D_AT_{\vec{n}} \varphi_A.
\end{eqnarray}
We will say that $\varphi_A$ is \textit{derived} from the solution $\mathcal{S}$.
This scaling function associated with matrix $A$ is generated by a solution $\mathcal{S}=\{s_\vn\}$ to the system of equations \eqref{eq:lawton}.

By Mallat's idea \cite{mallat}, a (orthogonal) scaling function (in $L^2(\R)$ when $A$ is the $1 \times 1$ matrix $[2]$) yields a (orthogonal) Parseval frame wavelet $\psi_A$.
A classical work \cite{db} by Daubechies provided a deep insight on this and related topics including the closed formula for related wavelet function $\psi_A$.
Lawton \cite{lawton} shows how to get a compact supported Parseval frame wavelet from a solution to the system of equations \eqref{eq:lawton} when $d=1$.
Bownik \cite{B} proved that the sequence defined in \eqref{eq:sequence} converges in $L^2(\R^d)$ in general,
the limit $\varphi_A$ is a scaling function and it produces a Parseval frame wavelet $\psi_A$ associated with matrix $A$.

The scaling function $\varphi$ in this paper is not necessarily orthogonal. So the wavelets and scaling functions we discuss in this paper fit the definition of the
\textit{frame multi-resolution analysis} (FMRA) by   J. Benedetto and S. Li \cite{benedetto} and  it also fits the definition of the \textit{general multi-resolution analysis} (GMRA) by
L. Baggett, H. Medina and K.  Merrill \cite{baggett}.
Many authors worked on the construction of Parseval wavelets with coefficients satisfying the two-scale relation. This includes more general unitary extension principle (UEP) of Ron and Shen \cite{DHRS} \cite{RS}.

We say that two scaling functions $\varphi_A$ and $\varphi_B$
are \textit{algebraically isomorphic} (see Definition \ref{def:PHIISO}), $\varphi_A \simeq \varphi_B$,
if they have matching solutions to their (reduced) isomorphic systems of equations (see Definition \ref{def:EQNISO}).

\begin{definition}
Let $\mathcal{W}_0(A,d)$ be the collection of all scaling functions in $L^2 (\R^d)$ associated with a given $d \times d$ \thematrix matrix $A$
and generated by finite solutions to \eqref{eq:lawton}.
Let $\mathcal{W}_0(d)$ be the union of $\mathcal{W}_0(A,d)$ for all $d\times d$ \thematrix matrix $A$.
Let $\mathcal{W}_0$ be the union of $\mathcal{W}_0(d)$ for $d\geq 1$.
\end{definition}

Our main result in this paper is:
\begin{theorem}\label{thm:newmain}
Let $A$ and $B$ be $d\times d$ and $s\times s$  \thematrix matrices with $d, s\geq 1$  respectively and let
$\varphi_A$ be a scaling function in $\mathcal{W}_0(A,d)$. There always exists a scaling function $\varphi_B\in \mathcal{W}_0(B,s)$ such that
\begin{equation*}
    \varphi_B \simeq \varphi_A.
\end{equation*}
\end{theorem}
Scaling functions in  $\mathcal{W}_0$  have compact supports since they are generated by finite solutions.
B. Han and R. Jia made an extended discussion on this in \cite{hanjia}.

The key technique in the proof of Theorem \ref{thm:newmain} is provided by Lemma \ref{lm:add} which shows that there exists a pair of ``local lattice isomorphisms" between special bounded subsets of lattice $\Z^d$ and its corresponding bounded subsets of $\Z$.
This algebraic isomorphism is an equivalence relation on $\mathcal{W}_0$ (see Lemma \ref{lm:phiiso}),
hence an equivalence relation on $\mathcal{W}_0(A,d)$ and on $\mathcal{W}_0(d)$ as well.
So there is a natural one-to-one mapping from $\frac{\mathcal{W}_0(A,d)}{\simeq}$ onto $\frac{\mathcal{W}_0(B,s)}{\simeq}$.
Also, there is a natural one-to-one mapping from
$\frac{\mathcal{W}_0(d)}{\simeq}$ onto $\frac{\mathcal{W}_0(s)}{\simeq}$.

Q. Gu and D. Han \cite{guhan} proved that, if an integral expansive matrix is associated with single function \textit{orthogonal wavelets} with multi-resolution analysis (MRA), then the matrix must be dyadic.  Our results could be generalized to multi-wavelets (see C. Cabrelli, C. Heil and U. Molter \cite{heil}).

This paper is organized as follows. In Section \ref{ss:lemmasdefinitions}, we collect general definitions, Lemmas and Propositions. In Section \ref{ss:base}, we present a special basis for $\Z^d$ that simplifies the format of $A\Z^d$. Section \ref{ss:core} is devoted to the proof of Theorem \ref{thm:newmain}.
In Section \ref{ss:notiso}, we present an example which shows that in Theorem \ref{thm:newmain} the condition on finiteness of the solutions can not be removed.

\section{Definitions and Lemmas}\label{ss:lemmasdefinitions}

\begin{definition}\label{def:red}
Let $A$ be a $d\times d$ expansive dyadic integral matrix and $\Lambda_A \in \Z^d$.
A \textit{reduced system of equations}  $\mathcal{E}_{(\Lambda_A,A,d)}$ is obtained from system of equations \eqref{eq:lawton} as following:
\begin{description}
  \item[Step 1] Replace all variables $h_\vn$ in \eqref{eq:lawton} by $0$, where $\vn\notin\Lambda_A$;
  \item[Step 2] Remove all trivial equations ``$0=0$".
\end{description}
\end{definition}
We denote the family of \textit{all} such reduced systems of equations by $\mathfrak{E}$.

For vector $\vk\in A\Z^d$,
if there exists at least one element vector $\vn\in\Lambda_A$ such that
the the vector $\vn+\vk$ is also in $\Lambda_A$,
then the equation $\sum_{\vn \in \Lambda_A} h_{\vn} \overline{h_{\vn+\vk}}=\delta_{\vec{0}\vk}$ is non-trivial and hence is in $\mathcal{E}_{(\Lambda_A,A,d)}$.
We say this equation is \textit{generated} by $\vk$.
A solution to the system has support \textit{contained} in $\Lambda_A$.

It is clear we have:
\begin{lemma}\label{lm:vklambda}
The following statements are equivalent:
\begin{description}
  \item[(A)] $\vk\in A\Z^d$ generates a non-trivial equation $\displaystyle\sum_{\vn \in \Lambda_A} h_{\vn} \overline{h_{\vn+\vk}}=\delta_{\vec{0}\vk}$ in $\mathcal{E}_{(\Lambda_A,A,d)}$.
  \item[(B)] $\Lambda_A \cap (\Lambda_A - \vk)\neq \emptyset$.
  \item[(C)] $\Lambda_A \cap (\Lambda_A + \vk)\neq \emptyset$.
\end{description}
\end{lemma}

We say that two elements $\vk_1,\vk_2 \in A\Z^d$ are related, $\vk_1 \sim \vk_2$, if the equation
$\sum_{\vn \in \Lambda_A} h_{\vn} \overline{h_{\vn+\vk_1}}=\delta_{\vec{0},\vk_1}$
is the same as either
$\sum_{\vn \in \Lambda_A} h_{\vn} \overline{h_{\vn+\vk_2}}=\delta_{\vec{0},\vk_2},$
or its conjugate
$\sum_{\vn \in \Lambda_A} \overline{h_{\vn}} h_{\vn+\vk_2}=\delta_{\vec{0},\vk_2}$.
This is an equivalence relation on $A\Z^d$ and it leads to a partition of $A\Z^d$.
We denote the partition by $\P_{\Lambda_A} \equiv \{P_{\emptyset,\Lambda_A}, P_{0,\Lambda_A}, P_{1,\Lambda_A}, \cdots \}$.
\begin{equation*}
A\Z^d = \big(P_{\emptyset,\Lambda_A}\dotcup P_{0,\Lambda_A}\big) \dotcup \bigcup \! \! \! \! \! \!   \cdot \ \  \{P_{j,\Lambda_A}\mid j>0, P_{j,\Lambda_A} \in \P_{\Lambda_A}\}.
\end{equation*}
$P_{\emptyset,\Lambda_A}$ contains all elements that generate the trivial equation $``0=0"$.
It is an infinite set.
$P_{0,\Lambda_A}$ contains only one element $\vec{0}$, which generates the equation $\sum_{\vn \in \Lambda_A} |h_{\vn}|^2 =1$.
Each set $P_{j,\Lambda_A}, j>0$, collects all elements that generate the same equation in $\mathcal{E}_{(\Lambda_A,A,d)}$. In some special cases of $\Lambda_A$, the part $\dotcup \{P_{j,\Lambda_A}\mid j>0, P_{j,\Lambda_A} \in \P_{\Lambda_A}\}$ could be empty. However
the next Lemma \ref{lm:pmk} shows that each set $P_{j,\Lambda_A}, j>0$, if nonempty,  has exactly two elements.

\begin{lemma}\label{lm:pmk}
If $\vk_1,\vk_2\in A\Z^d$ generate non-trivial equations in $\mathcal{E}_{(\Lambda_A,A,d)}$, then $\vk_1 \sim \vk_2$ if and only if $\vk_1 = \pm \vk_2$.
\end{lemma}
\begin{proof}
Let $\vk_1,\vk_2\in A\Z^d$ generate non-trivial equations in $\mathcal{E}_{(\Lambda_A,A,d)}$.

($\Rightarrow$) If $\vk_1 \sim \vk_2$ , then we have two cases:

(A) $\sum_{\vn \in \Lambda_A} h_{\vn} \overline{h_{\vn+\vk_1}}=\delta_{\vec{0},\vk_1}$ and $\sum_{\vn \in \Lambda_A} h_{\vn} \overline{h_{\vn+\vk_2}}=\delta_{\vec{0},\vk_2}$ are the same.
        Then there exists one common term $h_{\vn_1} \overline{h_{\vn_1+\vk_1}}$ and $h_{\vn_2} \overline{h_{\vn_2+\vk_2}}$ for some $\vn_1,\vn_2\in \Lambda_A$.
        So $\vn_1=\vn_2$ and $\vn_1+\vk_1 = \vn_2+\vk_2$, thus $\vk_1 = \vk_2$.

(B) $\sum_{\vn \in \Lambda_A} h_{\vn} \overline{h_{\vn+\vk_1}}=\delta_{\vec{0},\vk_1}$ and $\sum_{\vn \in \Lambda_A} \overline{h_{\vn}} h_{\vn+\vk_2}=\delta_{\vec{0},\vk_2}$ are the same.
        Then there exists one common term $h_{\vn_1} \overline{h_{\vn_1+\vk_1}}$ and $\overline{h_{\vn_2}} h_{\vn_2+\vk_2}$ for some $\vn_1, \vn_2 \in \Lambda_A$.
        So $\vn_1=\vn_2+\vk_2$ and $\vn_1+\vk_1 = \vn_2$, thus $\vk_1 = -\vk_2$.

($\Leftarrow$) It is obvious when $\vk_1 = \vk_2$.
So we only need to show that, if $\vk\in A\Z^d$ generates a non-trivial equation in $\mathcal{E}_{(\Lambda_A,A,d)}$, then $\vk \sim -\vk$, i.e. $\vk$ and $-\vk$ generate the same equation.
The equation generated by $\vk$ is
\begin{equation}\label{pf:eq1}
\sum_{\vn \in \Lambda_A \cap (\Lambda_A-\vk)} h_{\vn} \overline{h_{\vn+\vk}}=\delta_{\vec{0},\vk}.
\end{equation}
On the other hand, $-\vk$ generates
\begin{equation}\label{pf:eq2}
\sum_{\vn \in \Lambda_A \cap (\Lambda_A + \vk)} h_{\vn}\overline{h_{\vn-\vk}}=\delta_{\vec{0},-\vk}.
\end{equation}
Replacing $\vn = \vm + \vk$ in \eqref{pf:eq2}, we have $\sum_{\vm \in (\Lambda_A - \vk) \cap \Lambda_A} h_{\vm+\vk}\overline{h_{\vm}}=\delta_{\vec{0},-\vk}$.
The conjugate of this equation is
\[\sum_{\vm \in (\Lambda_A - \vk) \cap \Lambda_A} h_{\vm} \overline{h_{\vm+\vk}}=\delta_{\vec{0},-\vk}=\delta_{\vec{0},\vk},\]
which is the same as \eqref{pf:eq1}.
Hence $\vk \sim -\vk$.
\end{proof}

By the Axiom of Choice, there exists a choice function from $\{P_{j,\Lambda_A} \mid j\geq 0\} \rightarrow A\Z^d$ with range $\{\vec{0},\vk_1,\vk_2,...\}$,
where $\vk_j \in P_{j,\Lambda_A}$.
We call each of these choices an \textit{index set for the system of equations}.
It is clear that the following lemma characterizes index sets:
\begin{lemma}\label{lm:indexset}
A subset $E$ of $A\Z^d$ is an index set for the system  $\mathcal{E}_{(\Lambda_A,A,d)}$ if and only if it satisfies the following conditions (A)-(C):
\begin{description}
  \item[(A)] Every element $\vk$ in the set $E$  generates a (non-trivial) equation in  $\mathcal{E}_{(\Lambda_A,A,d)}$.
  \item[(B)] Two distinct elements $\vk_1,\vk_2$ in the set $E$ generate two distinct equations in  $\mathcal{E}_{(\Lambda_A,A,d)}$.
  \item[(C)] Each equation in  $\mathcal{E}_{(\Lambda_A,A,d)}$ has its generator in $E$.
\end{description}
\end{lemma}

Let $\Lambda_A^E$ be one of the index sets.
The system $\mathcal{E}_{(\Lambda_A,A,d)}$ has the form:
\begin{equation}\label{eq:reducedA}
\mathcal{E}_{(\Lambda_A,A,d)}: ~\left\{\begin{array}{l}
\sum_{\vec{n}\in \Lambda_A}h_{\vec{n}}\overline{h_{\vec{n}+\vec{k}}}=\delta_{\vec{0} \vec{k}},~ \vec{k}\in \Lambda_A^E,\\
\sum_{\vec{n}\in \Lambda_A}h_{\vec{n}}=\sqrt{2}.
\end{array}\right.
\end{equation}

\begin{lemma}\label{lm:subsetE}
For $\Lambda' \subseteq \Lambda\subset \Z^d$, every index set for $\mathcal{E}_{(\Lambda,A,d)}$ contains one and only one index set for $\mathcal{E}_{(\Lambda',A,d)}$.
\end{lemma}
\begin{proof}
Let $\Lambda^E$ be an index set for $\mathcal{E}_{(\Lambda,A,d)}$.
Given any equation in $\mathcal{E}_{(\Lambda',A,d)}$, if it is generated by an element $\vk \in A\Z^d$, then by Lemma \ref{lm:vklambda}, $\Lambda'\cap(\Lambda'-\vk)\neq \emptyset$.
Since $\Lambda' \subseteq \Lambda$, $\Lambda\cap(\Lambda-\vk)\neq \emptyset$.
This $\vk$ generates an equation in $\mathcal{E}_{(\Lambda,A,d)}$, by Lemma \ref{lm:vklambda} again.
Hence by Lemma \ref{lm:pmk}, $\Lambda^E$ contains $\vk$ or $-\vk$, but not both.
Collect all elements in $\Lambda^E$ that generate equations in $\mathcal{E}_{(\Lambda',A,d)}$.
This is an index set for $\mathcal{E}_{(\Lambda',A,d)}$.

\end{proof}

Let $B$ be an $s\times s$ expansive dyadic integral matrix, and $\Lambda_B \subset \Z^s$.
And
\begin{align*}
\mathcal{E}_{(\Lambda_B,B,s)}: ~\left\{\begin{array}{l}
\sum_{\vm\in \Lambda_B}h'_{\vm}\overline{h'_{\vm+\vl}}=\delta_{\vec{0} \vl},~ \vl\in \Lambda_B^E,\\
\sum_{\vm\in \Lambda_B}h'_{\vm}=\sqrt{2}.
\end{array}\right.
\end{align*}

\begin{definition}\label{def:EQNISO}
$\mathcal{E}_{(\Lambda_A,A,d)}, \mathcal{E}_{(\Lambda_B,B,s)} \in \mathfrak{E}$ are isomorphic,
$\mathcal{E}_{(\Lambda_A,A,d)}\sim\mathcal{E}_{(\Lambda_B,B,s)}$
 if there exist
\begin{description}
\item[(A)] a bijection $\theta: \Lambda_A \rightarrow \Lambda_B$ and
\item[(B)] a bijection $\eta$ from an index set $\Lambda_A^E$ of $\mathcal{E}_{(\Lambda_A,A,d)}$ onto an index set $\Lambda_B^E$ of $\mathcal{E}_{(\Lambda_B,B,s)}$
\end{description}
with the following properties:
for each $\vk\in \Lambda_A^E$, the equation in $\mathcal{E}_{(\Lambda_B,B,s)}$ generated by $\vl\equiv\eta(\vk)$ is obtained
by replacing $h_{\vn}$ by $h' _{\theta(\vn)}$ and $\delta_{\vec{0}\vk}$ by $\delta_{\vec{0}\vl}$ in the equation in $\mathcal{E}_{(\Lambda_A,A,d)}$ generated by $\vk$.
\end{definition}

\begin{proposition}\label{prop:equiv}
The isomorphism defined in Definition \ref{def:EQNISO} is an equivalence relation on $\mathfrak{E}$.
\end{proposition}

\begin{proof}
(Reflexivity) It is obvious that the relation is reflexive.

(Symmetry) Let $\mathcal{E}_{(\Lambda_A,A,d)} \sim \mathcal{E}_{(\Lambda_B,B,s)}$ with $\theta,\eta$ as the bijections.
It is easy to verify that $\theta^{-1},\eta^{-1}$ are the bijections for $\mathcal{E}_{(\Lambda_B,B,s)} \sim \mathcal{E}_{(\Lambda_A,A,d)}$.

(Transitivity) Let $\mathcal{E}_{(\Lambda_A,A,d)}, \mathcal{E}_{(\Lambda_B,B,s)}, \mathcal{E}_{(\Lambda_C,C,t)} \in \mathfrak{E}$,
and $\mathcal{E}_{(\Lambda_A,A,d)} \sim \mathcal{E}_{(\Lambda_B,B,s)}$ with $\theta_1,\eta_1$ as the bijections,
$\mathcal{E}_{(\Lambda_B,B,s)} \sim \mathcal{E}_{(\Lambda_C,C,t)}$ with $\theta_2,\eta_2$ as the bijections.
The reader can verify that the mappings $\theta \equiv \theta_2 \circ \theta_1$ and $\eta \equiv \eta_2 \circ \eta_1$ are the bijections from $\Lambda_A$ to $\Lambda_C$ and from $\Lambda_A^E$ to $\Lambda_C^E$, respectively.
Then when we replace variables $h_{\vn}$ by $h' _{\theta(\vn)}$ and replace $\delta_{\vec{0} \vk}$ by $\delta_{\vec{0},  \eta(\vec{k})}$ in equations of $\mathcal{E}_{(\Lambda_A,A,d)}$, we will obtain all equations in $\mathcal{E}_{(\Lambda_C,C,t)}$.

\end{proof}

The next Proposition \ref{prop:iso} provides sufficient conditions for this isomorphism.

\begin{proposition}\label{prop:iso}
Let $\mathcal{E}_{(\Lambda_A,A,d)}$ and $\mathcal{E}_{(\Lambda_B,B,s)}$ be elements in $\mathfrak{E}$ with index sets $\Lambda_A^E$ and $\Lambda_B^E$, respectively.
Assume there exist bijections
\[\theta:\Lambda_A \rightarrow \Lambda_B \text{ and } \eta: \Lambda_A^E \rightarrow \Lambda_B^E,\]
with properties that an extension of $\theta$ to new domain $\Z^d$ (denoted as $\Theta$) satisfies
\begin{equation*}
\Theta(\vn+\vk) = \theta(\vn) + \eta(\vk),~ \forall\vk \in \Lambda_A ^E \text{ and }\forall\vn \in \Lambda_A.
\end{equation*}
Then $\mathcal{E}_{(\Lambda_A,A,d)}$ and $\mathcal{E}_{(\Lambda_B,B,s)}$ are isomorphic.
\end{proposition}

Notice that this extension $\Theta$ is required to be neither one-to-one nor onto.

\begin{proof}
We will prove by definition.

(A) The equation $\Sigma_{\vec{n}\in \Lambda_A} h_{\vec{n}} = \sqrt{2}$ is in $\mathcal{E}_{(\Lambda_A,A,d)}$.
We replace $h_{\vn}$ by $h' _{\theta(\vn)}$ to obtain $\Sigma_{\vec{n}\in \Lambda_A} h'_{\theta(\vec{n})} = \sqrt{2}$.
Notice that $\theta(\Lambda_A) = \Lambda_B$,
we have the equation $\Sigma_{\vec{m}\in \Lambda_B} h'_{\vec{m}} = \sqrt{2}$ in $\mathcal{E}_{(\Lambda_B,B,s)}$.

(B) Let $\vk\in\Lambda_A^E$ and denote $\vl = \eta(\vk)\in \Lambda_B^E$.
$\vk$ generates the equation in $\mathcal{E}_{(\Lambda_A,A,d)}$
\[\sum_{\vn\in\Lambda_A}h_\vn  \overline{h_{\vn+\vec{k}}}=\delta_{\vec{0} \vk}.\]
Notice that the non-trivial products in this equation are those with indices $\vn,\vn+\vk \in \Lambda_A$.
By assumption,  $\theta(\vn + \vk)=\Theta(\vn + \vk)=\theta(\vn) + \eta(\vk) = \theta(\vn) +\vl$.
So replacing $h_{\vn}$ by $h' _{\theta(\vn)}$, $h_{\vn+\vk}$ by $h' _{\theta(\vn+\vk)}=h' _{\theta(\vn)+\vl}$ and $\delta_{\vec{0} \vk}$ by $\delta_{\vec{0}, \eta(\vk)}=\delta_{\vec{0} \vl}$,
the above equation becomes
\[\sum_{\vn\in\Lambda_A}h'_{\theta(\vn)}  \overline{h'_{\theta(\vn)+\vl}}=\delta_{\vec{0} \vl}.\]
This is the equation in $\mathcal{E}_{(\Lambda_B,B,s)}$ generated by $\vl$, i.e.
$\sum_{\vm\in\Lambda_B}h'_{\vm}  \overline{h'_{\vm+\vl}}=\delta_{\vec{0} \vl}$.

Since $\eta$ is an onto mapping, we obtained all equations in $\mathcal{E}_{(\Lambda_B,B,s)}$.

\end{proof}

\begin{definition}\label{def:PHIISO}
Let $\varphi_A$ be derived from $(\mathcal{S}_A, \mathcal{E}_{(\Lambda_A,A,d)})$ where $\mathcal{S}_A=\{a_\vn \mid \vn\in\Lambda_A\}$,
and $\varphi_B$ be derived from $(\mathcal{S}_B, \mathcal{E}_{(\Lambda_B,B,s)})$ where $\mathcal{S}_B=\{b_\vm \mid \vm\in\Lambda_B\}$.
$\varphi_A, \varphi_B $ are algebraically isomorphic, $\displaystyle \varphi_A \simeq \varphi_B $, if the following conditions are satisfied:
\begin{description}
  \item[(A)] $\mathcal{E}_{(\Lambda_A,A,d)}$ and $\mathcal{E}_{(\Lambda_B,B,s)}$ are isomorphic in $\mathfrak{E}$.
  \item[(B)] $b_{\theta(\vn)} = a _{\vn}, \forall \vn\in\Lambda_A$,
where $\theta$ is the bijection from $\Lambda_A$ onto $\Lambda_B$ in the isomorphism in (A).
\end{description}
\end{definition}
The reader will verify that following lemma is immediate from Definitions \ref{def:EQNISO} and \ref{def:PHIISO}.
\begin{lemma}\label{lm:phiiso}
For isomorphic systems $\mathcal{E}_{(\Lambda_A,A,d)}$ and $\mathcal{E}_{(\Lambda_B,B,s)}$ with bijection $\theta$ from $\Lambda_A$ to $\Lambda_B$,
if $\mathcal{S}_A=\{a_\vn \mid \vn\in\Lambda_A\}$ is a solution to $\mathcal{E}_{(\Lambda_A,A,d)}$,
then the set $\mathcal{S}_B \equiv \{b_\vm = a_{\theta^{-1}(\vm)} \mid \vm\in\Lambda_B\}$ is a solution to $\mathcal{E}_{(\Lambda_B,B,s)}$.
Moreover, the scaling functions derived from $(\mathcal{S}_A, \mathcal{E}_{(\Lambda_A,A,d)})$ and $(\mathcal{S}_B, \mathcal{E}_{(\Lambda_B,B,s)})$
are algebraically isomorphic.
\end{lemma}

Let $\varphi_0$ be a scaling function derived from the solution $\mathcal{S}_0 = \{a_\n \mid \n\in \Lambda_0\}$ to $\mathcal{E}_{(\Lambda_0,A,d)}$
with index set  $\Lambda_0^E$.
Let $\vn_0 \in \Z^d$,
 $\mathcal{S}_1 \equiv \{b_\vm=a_{\vm-\vn_0}\mid \vm \in \Lambda_0 +\vn_0\}$ and $\Lambda_1 \equiv \Lambda_0 +\vn_0$.
Then the set $\mathcal{S}_1$ is a solution to $\mathcal{E}_{(\Lambda_1,A,d)}$
with the same index set $\Lambda_0^E$.
We call the new scaling function $\varphi_1$ generated by $\mathcal{S}_1$ \textit {the scaling function of $\varphi_0$ after shifting $\vn_0$}.
It is easy to verify that the mappings $\theta$ and $\eta$ defined by $\theta(\vn) \equiv \vn +\vn_0, \vn \in \Lambda_0$ and $\eta(\vk) \equiv \vk, \vk\in \Lambda_0^E$ are the bijections that satisfy the conditions in Proposition \ref{prop:iso},
hence $\mathcal{E}_{(\Lambda_0,A,d)} \sim \mathcal{E}_{(\Lambda_1,A,d)}$ and $\varphi_0 \simeq \varphi_1$.
So we have

\begin{proposition}\label{prop:shift}
The scaling function of $\varphi_0$ after shifting $\vn_0$ is algebraically isomorphic to $\varphi_0$.
And their reduced system of equations are isomorphic.
\end{proposition}
Note that one can choose $\vn_0$ so that the new support contains $\vec{0}$.

\section{Matrix and Basis}\label{ss:base}

Let $d \geq 1$ be a natural number and $A$ a $d\times d$ expansive dyadic integral matrix.
We will need the following Smith Normal Form for integral matrices \cite{baeth}.
\begin{lemma}
For any expansive dyadic integral matrix $A$, we have $A = UDV$, where $U,V$ are integral matrices of determinant $\pm 1$, and $D$ a diagonal matrix with the last diagonal entry $2$ and all other diagonal entries $1$.
\end{lemma}

Let $\ve_1,...,\ve_d$ be the standard basis for $\Z^d$. Note that $V\Z^d = \Z^d$ and $U\Z^d = \Z^d$.
\begin{align*}
 \Z^d & = span\{\ve_1,...,\ve_{d-1},2\ve_d \}~ \dotcup ~\big( span\{\ve_1,...,\ve_{d-1},2\ve_d \}+\ve_d \big)\\
       & = D\Z^d ~\dotcup~ (D\Z^d + \ve_d) \\
      & = DV\Z^d ~\dotcup~ (DV\Z^d + \ve_d), \\
\Z^d = U\Z^d & = UDV\Z^d ~\dotcup~ U(DV\Z^d + \ve_d).
\end{align*}
So we have
\begin{proposition}\label{prop:partition}
Let $d\geq1$ be a natural number and $A$ a $d\times d$ expansive dyadic integral matrix. Then
\begin{equation*}
    \Z^d= A\Z^d  ~\dotcup~ (A\Z^d + U\ve_d).
\end{equation*}
\end{proposition}
Define $F = \{\vf_j = U \ve_j \mid j=1,...,d\}$.
This is a new basis for $\Z^d$ and matrix $U$ maps the standard basis to $F$.
We have
\begin{align*}
A\Z^d  & = UDV\Z^d  = UD\Z^d\\
       & = U~span\{\ve_1,...,\ve_{d-1},2\ve_d \} \\
       & =  span\{U\ve_1,...,U\ve_{d-1},2U\ve_d \}\\
       & = span \{ \vf_1,...,\vf_{d-1},2\vf_d \}.
\end{align*}
This proves
\begin{proposition}\label{prop:newbase}
Let $d\geq1$ be a natural number and $A$ a $d\times d$ expansive dyadic integral matrix. Then
$\R^d$ has a basis $\{\vf_j \mid j=1,...,d\}$ with properties that, under this new basis,
a vector $\vk$ is in $A\Z^d$ if and only if the last coordinate of $\vk$ is an even number.
That is, under this new basis, we have
\begin{align*}
A\Z^d = \{(\vx,2n) \mid \vx \in \Z^{d-1}, n\in\Z  \}.
\end{align*}
\end{proposition}

\section{Proof of Theorem \ref{thm:newmain}}\label{ss:core}

In this section, $d > 1$ and $N \geq 1$ are fixed natural numbers.
Let $A$ be a $d\times d$ expansive dyadic integral matrix.
By Proposition \ref{prop:newbase}, we use a new basis for $\R^d$ such that $A\Z^d = \{(\vx,2n) \mid \vx \in \Z^{d-1}, n\in\Z  \}$.

Define
\begin{align}\label{eq:Lambda_dn}
\Lambda_{d, N} & \equiv [0,2^N)^d \cap \Z^d = \big\{(n_1,\cdots,n_d) \mid 0\leq n_1,\cdots,n_d \leq 2^N-1\big\}.
\end{align}
The set $\Lambda_{d, N}$ contains  $2^{dN}$ elements in $\Z^d$.
So $\mathcal{E}_{(\Lambda_{d, N},A,d)}$ is a well-defined element in $\mathfrak{E}$.

For vector $\vn = (n_1,n_2,\cdots, n_{d-1},n_d) \in \Z^d$, define the function $\sigma_{_{d,N}}: \Z^d \rightarrow \Z$ as :
\begin{equation}\label{eq:sigmadN}
\sigma_{_{d,N}}(\vn) = \sum_{j=1}^d n_j \cdot 4^{(j-1)N}.
\end{equation}

\begin{lemma}\label{lm:sigma_inj}
Let $\vn\in (-2^N,2^N)^d \cap \Z^d.$ Then
\begin{description}
  \item[(A)] $\sigma_{_{d,N}}$ is an injection from $(-2^N,2^N)^d\cap\Z^d$ into $\Z$.
  \item[(B)] $\sigma_{_{d,N}}(\vn) = 0$ if and only if $\vn = \vec{0}$;
  \item[(C)] $\sigma_{_{d,N}}(\vn) > 0$ if and only if the right most non-zero coordinate of $\vn$ is positive;
  \item[(D)] $\sigma_{_{d,N}}(\vn) < 0$ if and only if the right most non-zero coordinate of $\vn$ is negative;
\end{description}
\end{lemma}
\begin{proof}
Let $\va,\vb$ be distinct elements in $(-2^N,2^N)^d\cap\Z^d$ and $j_0$ be the largest index where $a_{j_0} \neq b_{j_0}$.
Without loss of generality, we assume $a_{j_0}-b_{j_0} \geq 1$.
By definition of $\sigma_{_{d,N}}$ in \eqref{eq:sigmadN}, we have
\begin{align*}
\sigma_{_{d,N}}(\va) - \sigma_{_{d,N}}(\vb) 
& = (a_{j_0}-b_{j_0}) \cdot 4^{(j_0-1)N} + \sum_{j=1}^{j_0 - 1} (a_j - b_j) \cdot 4^{(j-1)N} \\
& \geq 4^{(j_0-1)N} - \sum_{j=1}^{j_0 - 1} |a_j - b_j| \cdot 4^{(j-1)N} \\
& \geq 4^{(j_0-1)N} - \sum_{j=1}^{j_0 - 1} (2^{N+1}-2) \cdot 4^{(j-1)N} \\
& =   4^{(j_0-1)N}(1 - \frac{2}{2^N+1}) + \frac{2}{2^N+1} > 0.
\end{align*}
Hence (A) is proved and it follows that (B) is also true.

Next, let $\va = (a_1, \cdots, a_d) \in(-2^N,2^N)^d\cap\Z^d$ and $j_0$ be the largest index where $a_j \neq 0$.
By definition of $\sigma_{_{d,N}}$, we have
\begin{align*}
\sigma_{_{d,N}}(\va) - a_{j_0}\cdot4^{(j_0-1)N}  = \sum_{j=1}^{j_0 - 1} a_j \cdot 4^{(j-1)N}.
\end{align*}
So
\[-\sum_{j=1}^{j_0 - 1} (2^N-1) \cdot 4^{(j-1)N}  \leq \sigma_{_{d,N}}(\va)  - a_{j_0}\cdot4^{(j_0-1)N} \leq \sum_{j=1}^{j_0 - 1} (2^N-1) \cdot 4^{(j-1)N}.\]
\[- \frac{4^{(j_0-1)N}-1}{2^N+1}  \leq \sigma_{_{d,N}}(\va)  - a_{j_0}\cdot4^{(j_0-1)N} \leq  \frac{4^{(j_0-1)N}-1}{2^N+1}.\]
\[ a_{j_0}\cdot4^{(j_0-1)N}  - \frac{4^{(j_0-1)N}-1}{2^N+1}  \leq \sigma_{_{d,N}}(\va) \leq  a_{j_0}\cdot4^{(j_0-1)N}  + \frac{4^{(j_0-1)N}-1}{2^N+1}.\]
It is clear that $\sigma_{_{d,N}}(\va)$ and $a_{j_0}$ have the same sign.
Hence (C) and (D) are proved.
\end{proof}

Define $f_{d,N}: \Z^d \rightarrow \Z$:
\begin{equation}\label{eq:fdN}
f_{d,N}(\vx,y) \equiv \big\lfloor\frac{y}{2}\big\rfloor 2^{(2d-3)N+2} + \left\{\begin{array}{ll}
2\sigma_{_{d-1,N}}(\vx)&\text{$y$ even}\\
2\sigma_{_{d-1,N}}(\vx)+1&\text{$y$ odd}\\
\end{array}\right.
~\forall \vx \in \Z^{d-1}, y \in \Z
\end{equation}
where $\lfloor\frac{y}{2}\rfloor$ gives the greatest integer that is less than or equal to $\frac{y}{2}$.
It is clear that $f_{d,N}$ has the following properties:
\begin{lemma}\label{lm:propf}
For $\vx, \vz\in \Z^{d-1}$ and $y,j \in \Z$,
\begin{description}
\item[(A)] $f_{d,N}(\vx+\vz,y) - f_{d,N}(\vx,y)  = 2\sigma_{_{d-1,N}}(\vz)$;
\item[(B)] $f_{d,N}(\vx,y+2j) - f_{d,N}(\vx,y)  = j\cdot2^{(2d-3)N+2}$;
\item[(C)] $f_{d,N}(\vx,y+1) - f_{d,N}(\vx,y)  = \left\{\begin{array}{ll}
                                              1 & \text{~y even;} \\
                                              2^{(2d-3)N+2}-1 & \text{~y odd.} \end{array}\right.$
\end{description}

\end{lemma}

Define
\begin{align}\label{eq:EdN}
E_{d,N} \equiv \{\vn=(\vx, 2j) \in \Z^d \mid \sigma_{_{d,N}}(\vn) \geq 0; \vn\in(-2^N,2^N)^d\cap\Z^d\}.
\end{align}

\begin{lemma}
$f_{d,N}$ defined in \eqref{eq:fdN} is an injection from $(-2^N,2^N)^d\cap\Z^d$ into $\Z$.
\end{lemma}
\begin{proof}
Let $(\vx+\vz,y+w),(\vx,y)$ be two distinct elements in $(-2^N,2^N)^d\cap\Z^d$.
Without loss of generality, we can further assume $w \geq 0$.
Denote $\Delta = f_{d,N}(\vx+\vz,y+w) - f_{d,N}(\vx,y)$.
We only need to show that $\Delta \neq 0$ when $\vz \neq \vec{0}$ or $w > 0$.
By Lemma \ref{lm:propf}, we have
\begin{align*}
\Delta & = 2\sigma_{_{d-1,N}}(\vz) + f_{d,N}(\vx,y+w) - f_{d,N}(\vx,y)\\
& = 2\sigma_{_{d-1,N}}(\vz) + \left\{\begin{array}{ll}
              j\cdot2^{(2d-3)N+2} & ~w=2j \\
              j\cdot2^{(2d-3)N+2} + f_{d,N}(\vx,y+1) - f_{d,N}(\vx,y)  & ~w=2j+1 \end{array}\right.
\end{align*}

When $w$ is even, we have the following two cases:

(A) $w=2j, j\geq1$. Notice that $2^{(2d-3)N+2} > 2\sigma_{_{d-1,N}}(\vz)$,
    so $\Delta = 2\sigma_{_{d-1,N}}(\vz) + j\cdot2^{(2d-3)N+2} > 0$.

(B) $w=0,\vz \neq \vec{0}$. $\Delta = 2\sigma_{_{d-1,N}}(\vz) \neq 0$, by Lemma \ref{lm:sigma_inj} (B).

When $w$ is odd, we have:
\begin{align*}
\Delta & = 2\sigma_{_{d-1,N}}(\vz) +j\cdot2^{(2d-3)N+2} + f_{d,N}(\vx,y+1) - f_{d,N}(\vx,y) \\
& = 2\sigma_{_{d-1,N}}(\vz) +j\cdot2^{(2d-3)N+2} + \left\{\begin{array}{ll}
                                              1 & y \text{ even} \\
                                              2^{(2d-3)N+2}-1 & y \text{ odd} \end{array}\right.
\end{align*}
So we have the following two cases:

(C) $w=2j+1, j=0$. Then
$\Delta = 2\sigma_{_{d-1,N}}(\vz) + \left\{\begin{array}{ll}
                                              1 & y \text{ even} \\
                                              2^{(2d-3)N+2}-1 & y \text{ odd} \end{array}\right.$
is an odd number, thus not zero for either case of $y$.

(D) $w=2j+1, j\geq1$. Then $\Delta > 2\sigma_{_{d-1,N}}(\vz) + j\cdot2^{(2d-3)N+2} >0$.

\end{proof}

$\Lambda_{d,N}$, $E_{d,N}$ are subsets of $(-2^N,2^N)^d\cap\Z^d$. So we have
\begin{lemma}\label{lm:fdN_inj}
$f_{d,N}$ defined in \eqref{eq:fdN} is an injection from $\Lambda_{d,N}$ into $\Z$ and is also an injection from $E_{d,N}$ into $\Z$.
\end{lemma}

\begin{proposition}\label{prop:lambda_dE}
The set $E_{d,N}$ defined in \eqref{eq:EdN} is an index set for $\mathcal{E}_{(\Lambda_{d, N},A,d)}$.
\end{proposition}
We define $\Lambda_{d,N}^E \equiv E_{d,N}$.
\begin{proof} We verify that this set $E_{d,N}$
satisfies the conditions in Lemma \ref{lm:indexset}. Let $\vk_0=(k_1, ..., k_{d-1},k_d)$ be a vector in  $E_{d,N}$.

(A) Define a vector $\vn_0=(n_1,...,n_d)$ by
$n_j=-k_j$ when $k_j \leq 0$ and $n_j=0$ when $k_j > 0$.
Then both $\vn_0, \vn_0+\vk_0 \in \Lambda_{d,N}$.
So each element in $E_{d,N}$ generates an equation in $\mathcal{E}_{(\Lambda_{d, N},A,d)}$ by Lemma \ref{lm:vklambda}.

(B) If $\vk_0=\vec{0}$, this vector generates the equation $\sum_{\vn \in \Lambda_{d, N}} |h_{\vn}|^2 =1$.
If  $\vk_0\neq\vec{0}$,
by Lemma \ref{lm:pmk} the only other vector generates the same equation as $\vk_0$ does is $-\vk_0$. This vector is not in $E_{d,N}$ since its right most coordinate is a negative number (Lemma \ref{lm:sigma_inj}).

(C) If $\vk$ is not in $(-2^N,2^N)^d\cap\Z^d$, then one of its coordinate's absolute value is greater than or equal to $2^N$.
For all $\vn\in \Lambda_{d,N}$, $\vn+\vk \notin \Lambda_{d,N}$,
so every $\vk$ that generates an equation in $\mathcal{E}_{(\Lambda_{d, N},A,d)}$ must be in $(-2^N,2^N)^d\cap\Z^d$, by Lemma \ref{lm:vklambda}.
By Lemma \ref{lm:pmk}, $\pm\vk$ will generate the same equation.
However, one of $\vk$ and $-\vk$ must be in $E_{d,N}$ by its definition. So $E_{d,N}$ generates all equations.

\end{proof}

Define mappings $\theta_{d,N}$ and $\eta_{d,N}$ as follows:
\begin{align}\label{eq:theta}
\theta_{d,N}\big((\vx,y)\big) & \equiv f_{d,N}(\vx,y), ~(\vx,y) \in \Lambda_{d,N} \\
\label{eq:eta}
\eta_{d,N}\big((\vx,y)\big) & \equiv f_{d,N}(\vx,y), ~(\vx,y) \in \Lambda_{d,N}^E.
\end{align}
By Lemma \ref{lm:fdN_inj}, $\theta_{d,N},\eta_{d,N}$ are injections on $\Lambda_{d,N}$ and $\Lambda_{d,N}^E$ respectively.

\begin{remark}\label{rm:remark}
Given a $d \times d$ expansive dyadic integral matrix $A$,
the new basis $F = \{\vf_j \mid j=1,...,d\}$ as discussed in Section \ref{ss:base} Proposition \ref{prop:newbase} is not unique in general.
The related sets $\Lambda_{d,N},\Lambda_{d,N}^E$, the mappings $\theta_{d,N},\eta_{d,N}$ and then the images $\theta_{d,N}(\Lambda_{d,N}), \eta_{d,N}(\Lambda_{d,N}^E)$
are all well-defined and are related to the matrix $A$ as well as the new basis.
Thus different basis gives different sets, mappings and images.
However, for the sake of simplicity of notation system, we will use the above notations without mentioning $A$ or the new basis $F$.
\end{remark}

\begin{lemma}\label{lm:add}
Let $\vk\in\Lambda_{d,N} ^E$ and $\vn\in\Lambda_{d,N}.$ Then
\begin{equation*}
    f_{d,N} (\vn+\vk) = \theta_{d,N}(\vn) + \eta_{d,N} (\vk).
\end{equation*}
\end{lemma}
\begin{proof}
Denote $\vn = (\vu,n_d), \vk = (\vv,k_d)$ where $\vu,\vv \in \Z^{d-1}$ and $k_d\in2\Z$.
\begin{align*}
\eta_{d,N}(\vk) + \theta_{d,N}(\vn)  = & \left\lfloor\frac{k_d}{2}\right\rfloor 2^{(2d-3)N+2} +  2\sigma_{_{d-1,N}}(\vv)\\
& + \left\lfloor\frac{n_d}{2}\right\rfloor 2^{(2d-3)N+2} + \left\{\begin{array}{ll}
2\sigma_{_{d-1,N}}(\vu)&\text{$n_d$ even;}\\
2\sigma_{_{d-1,N}}(\vu)+1&\text{$n_d$ odd,}\\
\end{array}\right. \\
 = &\left\lfloor\frac{n_d+k_d}{2}\right\rfloor 2^{(2d-3)N+2} + \left\{\begin{array}{ll}
2\sigma_{_{d-1,N}}(\vu+\vv)&\text{$n_d$ even;}\\
2\sigma_{_{d-1,N}}(\vu+\vv)+1&\text{$n_d$ odd,}\\
\end{array}\right.\\
 =& f_{d,N}(\vn + \vk).
\end{align*}
\end{proof}

Define
\begin{equation}\label{eq:theta_lambda_d}
\Lambda_{1,N} \equiv \theta_{d,N}(\Lambda_{d,N}).
\end{equation}
$\Lambda_{1,N}$ is a subset of $\Z$ and $\mathcal{E}_{(\Lambda_{1,N},[2],1)}$ is an element in $\mathfrak{E}$.
\begin{lemma}\label{lm:limits}
\begin{description}
  \item[(A)] $\displaystyle\min(\Lambda_{1,N}) = 0.$
  \item[(B)] $\displaystyle\max(\Lambda_{1,N}) = (2^N-1)2^{(2d-3)N+2} + \frac{2}{2^N+1}(4^{(d-1)N})+ 1$.
  \item[(C)]\label{it:limits3} $\displaystyle \min \big(\eta_{d,N} (\Lambda_{d,N} ^E)\big) =0$.
  \item[(D)]\label{it:subset}  $\Lambda_{1,N}$ contains the consecutive integers $\{0,1,\cdots, 2^{N+1}-1\}$.
\end{description}
\end{lemma}

\begin{proposition}\label{prop:eta_lambda_dE}
$\eta_{d,N}(\Lambda_{d , N} ^E)$ is an index set for $\mathcal{E}_{(\Lambda_{1,N},[2],1)}$.
\end{proposition}
We denote $\eta_{d,N}(\Lambda_{d , N} ^E)$ as $\Lambda_{1,N} ^E$.
\begin{proof} We will verify that the set $\eta_{d,N}(\Lambda_{d , N} ^E)$ satisfies the three conditions in Lemma \ref{lm:indexset}.

(A) Let $\ell=\eta_{d,N}(\vk)$ for some $\vk \in \Lambda_{d , N} ^E$.
By Lemma \ref{lm:vklambda}, $\Lambda_{d,N} \cap (\Lambda_{d,N} - \vk) \neq \emptyset$.
Let $\vn_0\in\Lambda_{d,N} \cap (\Lambda_{d,N} - \vk)$.
So $\theta_{d,N}(\vn_0)\in \Lambda_{1,N}.$ Since $\vn_0\in \Lambda_{d,N}-\vk$,
$\vn_0+\vk\in\Lambda_{d,N}$ hence $\theta_{d,N}(\vn_0+\vk)\in \Lambda_{1,N}$.
By Lemma \ref{lm:add}, $\theta_{d,N}(\vn_0+\vk)=f_{d,N}(\vn_0+\vk)=\theta_{d,N}(\vn_0)+\theta_{d,N}(\vk)=\theta_{d,N}(\vn_0)+\ell$.
This implies
$\theta_{d,N}(\vn_0)\in\Lambda_{1,N}-\ell$.
So
\[\Lambda_{1,N} \cap (\Lambda_{1,N}-\ell) \neq \emptyset.\]
By Lemma \ref{lm:vklambda}, every element $\ell \in \eta_{d,N}(\Lambda_{d , N} ^E)$ generates an equation in $\mathcal{E}_{(\Lambda_{1,N},[2],1)}$.

(B) Let $\ell_1, \ell_2 $ be distinct elements in $\eta_{d,N}(\Lambda_{d , N} ^E)$.
By part (A), $\ell_1$ and $\ell_2$ generate equations in $\mathcal{E}_{(\Lambda_{1,N},[2],1)}$.
By Lemma \ref{lm:limits} (c), $\ell_1, \ell_2$ are non-negative, hence $\ell_1 \neq \ell_2$.
So by Lemma \ref{lm:pmk}, the two elements $\ell_1, \ell_2$ must generate different equations.

(C) Given an equation in $\mathcal{E}_{(\Lambda_{1,N},[2],1)}$ generated by an element $\ell \in 2\Z$.
By Lemma \ref{lm:pmk}, we can further assume that $\ell \geq 0$.
It suffices to show that $\ell\in \eta_{d,N}(\Lambda_{d,N} ^E)$.
By Lemma \ref{lm:vklambda}, we have $\Lambda_{1,N}\cap(\Lambda_{1,N} - \ell) \neq \emptyset$.
There exists a vector $n_0 \in \Lambda_{1,N}$ such that $n_0 + \ell \in \Lambda_{1,N}$.
By definition of $\Lambda_{1,N}$, there exist
$\vn_1 = (\vx_1,y_1) \in \Lambda_{d,N}$ and $\vn_2 = (\vx_2,y_2) \in \Lambda_{d,N}$ such that
\begin{align*}
n_0 & = f_{d,N}(\vn_1),\\
n_0 + \ell & = f_{d,N}(\vn_2).
\end{align*}
Since $\ell$ is even, by Equation  \eqref{eq:fdN} $y_2-y_1$ must be even.

We have
\begin{align*}
\ell & = f_{d,N}(\vn_2) - f_{d,N}(\vn_1)\\
     & = f_{d,N}\big((\vx_2,y_2)\big) - f_{d,N}((\vx_1,y_1)\big)\\
     & = f_{d,N}\big((\vx_2,y_2 - y_1 + y_1)\big) - f_{d,N}((\vx_1,y_1)\big)\\
     & = \big\lfloor\frac{y_2-y_1}{2}\big\rfloor 2^{(2d-3)N+2} + f_{d,N}\big((\vx_2, y_1)\big) - f_{d,N}((\vx_1,y_1)\big)\\
     & = \big\lfloor\frac{y_2-y_1}{2}\big\rfloor 2^{(2d-3)N+2} + 2\sigma_{_{d-1,N}}(\vx_2 - \vx_1)\\
     & = f_{d,N}\big((\vx_2-\vx_1,y_2-y_1)\big)
\end{align*}
Denote $\vk = (\vx_2-\vx_1,y_2-y_1)$.
It is clear that $\vk \in (\Lambda_{d,N}-\Lambda_{d,N}) \subseteq (-2^N,2^N)^d\cap\Z^d$.
It is left to show that $\sigma_{_{d,N}}(\vk) \geq 0$.
Since $f_{d,N}(\vk) = \ell \geq 0$,  by \eqref{eq:fdN} we must have $y_2-y_1 \geq 0$.
If $y_2-y_1>0$, by Lemma \ref{lm:sigma_inj} $\sigma_{_{d,N}}(\vk) \geq 0$.
If $y_2-y_1=0$, then the above calculation gives $\sigma_{_{d,N}}(\vk) = \sigma_{_{d-1,N}}(\vx_2-\vx_1) = \frac{1}{2}f_{d,N}(\vk) = \frac{\ell}{2} \geq 0$.
Proposition \ref{prop:eta_lambda_dE} is proved.

%

\end{proof}

We summarize the above discussions. For an integer $N \geq 1$, we first define $\Lambda_{d,N}$ in \eqref{eq:Lambda_dn}. So $\mathcal{E}_{(\Lambda_{d, N},A,d)}$
is well defined. We prove that the set defined in \eqref{eq:EdN} is an index set for $\mathcal{E}_{(\Lambda_{d, N},A,d)}$.
Secondly, we define mappings $\theta_{d,N},\eta_{d,N}$ in \eqref{eq:theta} and \eqref{eq:eta}, respectively.
Thirdly, we define $\Lambda_{1, N}$ in \eqref{eq:theta_lambda_d}.
So $\mathcal{E}_{(\Lambda_{1, N},[2],1)}$ is well-defined.
At last, we prove that $\Lambda_{1, N}^E \equiv \eta_{d,N}(\Lambda_{d , N} ^E)$ is an index set for $\mathcal{E}_{(\Lambda_{1,N},[2],1)}$ in Proposition \ref{prop:eta_lambda_dE}.
By Lemma \ref{lm:fdN_inj}, $\theta_{d,N}$ is a bijection from $\Lambda_{d,N}$ onto $\Lambda_{1,N}$,
and $\eta_{d,N}$ is a bijection from $\Lambda_{d,N} ^E$ onto $\Lambda_{1,N} ^E$.
Therefore, by Lemma \ref{lm:add} and  Proposition \ref{prop:iso} we have

\begin{proposition}\label{prop:iso_sq}
The systems of equations $\mathcal{E}_{(\Lambda_{d, N},A,d)}$ and $\mathcal{E}_{(\Lambda_{1, N},[2],1)}$ are isomorphic,
\begin{equation*}
    \mathcal{E}_{(\Lambda_{d, N},A,d)} \sim \mathcal{E}_{(\Lambda_{1, N},[2],1)}.
\end{equation*}
\end{proposition}

Next, let $\Lambda_d$ be a non-empty subset of $\Lambda_{d, N}$.
By Lemma \ref{lm:subsetE},  $\Lambda_{d,N}^E$ contains an index set $\Lambda_d^E$ for system $\mathcal{E}_{(\Lambda_d,A,d)}$.
Define $\Lambda_1 \equiv \theta_{d,N} (\Lambda_d)$.
The mapping $\theta_{d,N}$ is a bijection from $\Lambda_d$ onto $\Lambda_1$.
Since $\Lambda_1$ is a subset of $\Lambda_{1,N}$, by Lemma \ref{lm:subsetE} again, $\Lambda_{1,N}^E$ contains an index set for system $\mathcal{E}_{(\Lambda_1,[2],1)}$.

\begin{proposition}\label{prop:subsetlambda}
Let
$\Lambda_d$ be a non-empty subset of  $\Lambda_{d,N}$ and denote $\Lambda_1 \equiv \theta_{d,N} (\Lambda_d)$.
Then
\begin{equation*}
\mathcal{E}_{(\Lambda_d,A,d)} \sim \mathcal{E}_{(\Lambda_1,[2],1)}.
\end{equation*}
\end{proposition}
\begin{proof}
Since $\Lambda_d$ is a subset of $\Lambda_{d,N}$,
by Lemma \ref{lm:subsetE}, $\Lambda_{d,N}^E$ contains exactly one subset which is an index set for $\mathcal{E}_{(\Lambda_d,A,d)}$.
We denote it by $\Lambda_d^E$.
Also, since $\Lambda_1\equiv \theta_{d,N} (\Lambda_d)$ is a subset of $\Lambda_{1,N}$,
by Lemma \ref{lm:subsetE} again, $\Lambda_{1,N}^E$ contains exactly one subset which is an index set for $\mathcal{E}_{(\Lambda_1,[2],1)}$.
We denote it by $\Lambda_1^E$.

We claim the following equality holds:
\begin{equation*}
    \Lambda_1^E = \eta_{d,N} (\Lambda_d^E).
\end{equation*}
Let $\vk\in\Lambda_d^E$, then $\vk\in\Lambda_{d,N}^E$. So $\eta_{d,N}(\vk)\in\Lambda_{1,N}^E$.
Also, there exists $\vn_0\in\Lambda_d$ such that $\vn_0+\vk \in \Lambda_d$.
So $\theta_{d,N}(\vn_0)\in\theta_{d,N}(\Lambda_d)=\Lambda_1$ and $\theta_{d,N}(\vn_0+\vk) \in \theta_{d,N}(\Lambda_d)=\Lambda_1$.
By Lemma \ref{lm:add},  $\theta_{d,N}(\vn_0+\vk) = \theta_{d,N}(\vn_0) + \eta_{d,N}(\vk)\in \Lambda_1$.
Notice that $\theta_{d,N}(\vn_0)\in\Lambda_1$, then $\eta_{d,N}(\vk)$ is in an index set for $\mathcal{E}_{(\Lambda_1,[2],1)}$.
By Lemma \ref{lm:subsetE}, $\eta_{d,N}(\vk)$ must be in $\Lambda_1^E$,
i.e. $\eta_{d,N} (\Lambda_d^E) \subseteq  \Lambda_1^E$.

Let $\ell\in \Lambda_1^E$, then $\ell\in\Lambda_{1,N}^E$.
So there exists $\vk\in\Lambda_{d,N}^E$ such that $\ell=\eta_{d,N}(\vk)$.
Also, there exists $m_0\in\Lambda_1$ such that $m_0+\ell \in \Lambda_1$.
Let $\vn_0\in\Lambda_d$ such that $\theta_{d,N}(\vn_0)=m_0$.
We have
\begin{align*}
\theta_{d,N}(\vn_0+\vk) = \theta_{d,N}(\vn_0) + \eta_{d,N}(\vk) = m_0 + \ell \in \Lambda_1.
\end{align*}
So $\vn_0+\vk\in\Lambda_d$, hence $\vk$ is in an index set for $\mathcal{E}_{(\Lambda_d,A,d)}$.
By Lemma \ref{lm:subsetE}, $\vk\in\Lambda_d^E\subseteq\Lambda_{d,N}^E$.
Therefore $\Lambda_1^E \subseteq \eta_{d,N} (\Lambda_d^E)$.
The claim is proved.

So $\eta_{d,N}$ is an bijection from $\Lambda_d^E$ onto $\Lambda_1^E=\eta_{d,N}(\Lambda_d^E)$, when restricted to $\Lambda_d^E$.
And it is clear that $\theta_{d,N}$ is a bijection from $\Lambda_d$ onto $\Lambda_1$, when restricted to $\Lambda_d$.

By Lemma \ref{lm:add} $f_{d,N}$ is an extension of $\theta_{d,N}$ and
\begin{equation*}
    f_{d,N}(\vn+\vk) = \theta_{d,N}(\vn) + \eta_{d,N}(\vk), \forall \vk\in\Lambda_{d,N}^E\textrm{ and }\vn\in\Lambda_{d,N}.
\end{equation*}
Since $\Lambda_d^E$
and $\Lambda_d$ are subsets of $\Lambda_{d,N}^E, \Lambda_{d,N}$ respectively.
We have
\begin{equation*}
    f_{d,N}(\vn+\vk) = \theta_{d,N}(\vn) + \eta_{d,N}(\vk), \forall \vk\in\Lambda_d^E\textrm{ and }\vn\in\Lambda_d.
\end{equation*}
Hence $\mathcal{E}_{(\Lambda_d,A,d)} $ is isomorphic to $\mathcal{E}_{(\Lambda_1,[2],1)}$.
Proposition \ref{prop:subsetlambda} is proved.

\end{proof}

We will first prove the following theorem and this will lead to the proof of Theorem \ref{thm:newmain}.
\begin{theorem}\label{thm:main}
For every scaling function $\varphi_d$ in $\mathcal{W}_0 (d)$, there exists a scaling function $\varphi_1$ in
$\mathcal{W}_0 (1)$ such that
\begin{equation*}
    \varphi_1 \simeq \varphi_d.
\end{equation*}
\end{theorem}

\begin{proof}
Let $A$ be a $d \times d$ expansive dyadic integral matrix,
and $\varphi_d\in\mathcal{W}_0(A,d)$ be a scaling function derived from $(\mathcal{S}_d, \mathcal{E}_{(\Lambda_d,A,d)})$,
where $\mathcal{S}_d \equiv \{a_{\vn}, \vn\in \Lambda_d\}$ is a solution with finite support $\Lambda_d$.

(A) Define $\vn_0 \equiv (x_1,...,x_d) , x_j = \min\{n_j \mid \vn=(n_1,...,n_d)\in \Lambda_d\}$.
It's clear that all coordinates of elements in the set $\Lambda \equiv \Lambda_d - \vn_0$ are non-negative.
Define $\theta_0(\vn) \equiv \vn -\vn_0, \forall \vn \in \Lambda_d$ and $\eta_0(\vk) = \vk, \forall \vk \in \Lambda_d^E$.
By Proposition \ref{prop:shift}, we have $\mathcal{E}_{(\Lambda_d,A,d)} \sim \mathcal{E}_{(\Lambda,A,d)}$ with the related bijections $\theta_0$ and $\eta_0$.

(B) Let $N$ be a natural number such that $\Lambda_{d,N}$ contains $\Lambda$.
Let $\theta_{d,N}, \eta_{d,N}$ as defined in \eqref{eq:theta} and \eqref{eq:eta}, and $\Lambda_1 \equiv \theta_{d,N}(\Lambda)$.
By Proposition \ref{prop:subsetlambda}, $\mathcal{E}_{(\Lambda,A,d)} \sim \mathcal{E}_{(\Lambda_1,[2],1)}$ with associated bijections $\theta_{d,N}, \eta_{d,N}$.

By (A)(B) and Proposition \ref{prop:equiv}, $\mathcal{E}_{(\Lambda_d,A,d)} \sim \mathcal{E}_{(\Lambda_1,[2],1)}$.
Note that $\theta \equiv \theta_{d,N}\circ \theta_0,\eta\equiv \eta_{d,N}\circ \eta_0$ are the bijections in this isomorphism and $\Lambda_1 = \theta(\Lambda_d)$.
Define $\mathcal{S}_1\equiv \{b_\vm = a_{\theta^{-1}(\vm)} \mid \vm\in\Lambda_1\}$.
By Lemma \ref{lm:phiiso}, $\mathcal{S}_1$ is a solution to $\mathcal{E}_{(\Lambda_1,[2],1)}$.
Let $\varphi_1$ be the scaling function derived from $(\mathcal{S}_1,\mathcal{E}_{(\Lambda_1,[2],1)})$.
By Lemma \ref{lm:phiiso} again, we have
\[\varphi_d \simeq \varphi_1.\]
It is clear that $\varphi_1 \in  \mathcal{W}_0 (1)$.
Theorem \ref{thm:main} is proved.

\end{proof}

\begin{theorem}\label{thm:main_dual}
Let $B$ be an $s \times s$ expansive dyadic integral matrix.
For every scaling function $\varphi_1$ in $\mathcal{W}_0(1)$, there exists a scaling function $\varphi_s$ in $\mathcal{W}_0 (B,s)$ associated with $B$ such that
\begin{equation*}
    \varphi_1 \simeq \varphi_s.
\end{equation*}
\end{theorem}

\begin{proof}
$\varphi_1$ is derived from a solution $\mathcal{S}_1=\{a_\vn \mid \vn\in\Lambda_1\}$ to $\mathcal{E}_{(\Lambda_1,[2],1)}$ where $\Lambda_1$ is a finite set.
By Proposition \ref{prop:shift} $\varphi_1$ is algebraically isomorphic to the scaling function after shifting by $\min \Lambda_1$.
Without loss of generalization we will simply assume that $\min \Lambda_1 =0$.
Let $B$ be an $s \times s$ expansive dyadic integral matrix.
By Proposition \ref{prop:newbase}, we can choose a basis for $\R^s$ such that
\[B\Z^s = \{(\vx,2n) \mid \vx \in \Z^{s-1}, n\in\Z  \}.\]
Let $N'$ be the smallest natural number such that
\begin{equation}\label{eq:bound}
\max(\Lambda_1) \leq 2^{N'+1}-1.
\end{equation}
Define $\Lambda_{s,N'}\equiv [0,2^{N'})^s \cap \Z^s $ as in \eqref{eq:Lambda_dn}.
By Proposition \ref{prop:lambda_dE} the set
\begin{equation*}
\Lambda_{s,N'}^E \equiv \{\vn=(\vx, 2j) \in \Z^s \mid \sigma_{_{s,N'}}(\vn) \geq 0; \vn\in(-2^{N'},2^{N'})^s\cap\Z^s\}
\end{equation*}
is an index set for the system of equations $\mathcal{E}_{(\Lambda_{s,N'},B,s)}$.
Define
\begin{equation*}
    \Lambda_{1,N'} \equiv \theta_{s,N'}(\Lambda_{s,N'}).
\end{equation*}
By Proposition
\ref{prop:iso_sq}
\begin{equation*}
    \mathcal{E}_{(\Lambda_{s, N'},B,s)}\sim\mathcal{E}_{(\Lambda_{1, N'},[2],1)}
\end{equation*}
with well-defined bijections $\theta_{s,N'}$ and $\eta_{s,N'}$.

By Lemma \ref{lm:limits} (D) and \eqref{eq:bound}, the set $\Lambda_1$ is a subset of $\Lambda_{1,N'}$.
Define
\begin{equation*}
    \Lambda_s \equiv \theta_{s,N'} ^{-1} (\Lambda_1).
\end{equation*}
It is clear that $\Lambda_s$ is a subset of $\Lambda_{s,N'}.$ By Proposition \ref{prop:subsetlambda}
\begin{equation*}
    \mathcal{E}_{(\Lambda_s,B,s)}\sim\mathcal{E}_{(\Lambda_1,[2],1)}
\end{equation*}
with bijections which are $\theta_{s,N'}$ and $\eta_{s,N'}$ restricted to $\Lambda_s$ and $\Lambda_1$ respectively.

By Lemma \ref{lm:phiiso}, the set $\mathcal{S}_s \equiv \{b_\vm = a_{\theta_{s,N'}^{-1}(\vm)} \mid \vm\in\Lambda_s\}$ is a solution to $\mathcal{E}_{(\Lambda_s,B,s)}$,
and the scaling function $\varphi_s$ derived from $(\mathcal{S}_s,\mathcal{E}_{(\Lambda_s,B,s)})$ and $\varphi_1$ are algebraically isomorphic.
It is clear that $\varphi_s \in \mathcal{W}_0 (B,s)$.
Theorem \ref{thm:main_dual} is proved.
\end{proof}

\textsc{Proof of Theorem \ref{thm:newmain}}
Let $A$ be a $d\times d$ expansive dyadic integral matrix with $d\geq 1$ and $\varphi_A$ be a scaling function in $\mathcal{W}_0(d)$
associated with $A$.
By Theorem \ref{thm:main}, there exists a scaling function $\varphi_1$ in $\mathcal{W}_0(1)$
such that $\varphi_A \simeq \varphi_1$.

Let $B$ be an $s\times s$ expansive dyadic integral matrix with $s\geq 1$.
By Theorem \ref{thm:main_dual} there is a scaling function $\varphi_B\in \mathcal{W}_0(B,s)$ associated with $B$ such that $\varphi_1 \simeq \varphi_B$.
Since the relation $\simeq$ in $\mathcal{W}_0$ is an equivalence relation and hence transitive, so
\begin{equation*}
    \varphi_B \simeq \varphi_A.
\end{equation*}
Theorem \ref{thm:newmain} is proved.

\section{Discussion}\label{ss:notiso}

In this section, we  provide an example of (orthogonal) scaling function $\varphi_c$ in $L^2(\R^2)$ with the property that the support of its associated solution, $\Lambda_c$ is infinite and also that $\varphi_c$ is not algebraically isomorphic to any scaling function in $L^2 (\R)$. So the finiteness assumption (on scaling functions in $\mathcal{W}_0$) in Theorem \ref{thm:newmain} can not be removed.

\begin{proposition}\label{prop:noiso}
Let $A$ be a $2\times 2$ dyadic integral expansive matrix and
let $\varphi_0$ be a scaling function derived from a solution $\{a_\vn \}$ to
$\mathcal{E}_{(\Lambda_2,A,2)}$. Assume further that the support
$\Lambda_2$ contains the set $L$, a sublattice of $\Z^2$,
\begin{equation*}
    L\equiv \{\vn = (m,n) \in \Z^2 \mid m+n\in 2\Z\}.
\end{equation*}
Then there is no scaling function in $L^2(\R)$  algebraically isomorphic to $\varphi_0$.
\end{proposition}

\begin{proof}

We prove by contradiction.
Let $\varphi_1$ be a scaling function in $L^2(\R)$ derived from its solution $\{b_n \mid n\in \Lambda_1\}$ to its reduced system of equations $\mathcal{E}_{(\Lambda_1,[2],1)}$ with the index set $\Lambda_1^E$.
Assume that $\varphi_0 \simeq \varphi_1 $.
By definition we have
$\mathcal{E}_{(\Lambda_2,A,2)} \sim \mathcal{E}_{(\Lambda_1,[2],1)}$ with bijections $\theta$ and $\eta$, $\Lambda_1=\theta(\Lambda_2)$ and $\eta(\Lambda_2^E) = \Lambda_1^E$.
Without loss of generality, by Proposition \ref{prop:shift} we can further assume that
$\vec{0}$ is in $\Lambda_2$ and $0$ is in $\Lambda_1$, and that $\theta$ maps $\vec{0}$ to $0$.

Note that since $L=\{\vn = (m,n) \in \Z^2 \mid m+n\in 2\Z\}\subset \Lambda_2$, by Lemma \ref{lm:vklambda}
each element $\vl$ in $E\equiv A(2\Z^2)$ will generate an equation
in the reduced system of equations of $\mathcal{E}_{(\Lambda_2,A,2)}$ hence by Lemma \ref{lm:pmk}
one of $\pm \vl$ must be in $\Lambda_2^E$. Since $E\subset L$ and $L$ is a sublattice of $\Z^2$, we have
\begin{equation*}
\vn+ n \vk \in L \subseteq \Lambda_2, \forall n \in\Z, \forall\vk \in E\cap \Lambda_2^E \textrm{ and  }\forall \vn \in L.
\end{equation*}

Let $\vk \in E\cap \Lambda_2^E$. Since $\mathcal{E}_{(\Lambda_2,A,2)} \sim \mathcal{E}_{(\Lambda_1,[2],1)}$, the equation
\begin{equation}\label{eq:k2}
\sum_{\vn\in\Lambda_2} h_{\vn}\overline{h_{\vn+\vk}} = \delta_{\vec{0}\vk}
\end{equation}
in $\mathcal{E}_{(\Lambda_2,A,2)}$ generated by $\vk$, corresponds to
\begin{equation}\label{eq:k1a}
\sum_{m\in \Lambda_1=\theta(\Lambda_2)} h'_{m}\overline{h'_{m+\eta(\vk)}} = \delta_{0\eta(\vk)}
\end{equation}
in $\mathcal{E}_{(\Lambda_1,[2],1)}$ generated by $\eta(\vk) \in \Lambda_1^E$.
On the other hand, by Definition \ref{def:EQNISO}, when we replace $\vn,\vn+\vk$ by $\theta(\vn),\theta(\vn+\vk)$ and $\delta_{0\vk}$ by $\delta_{0\eta(\vk)}$ in Equation \eqref{eq:k2}, we obtain
\begin{equation}\label{eq:k1b}
\sum_{\vn\in\Lambda_2} h'_{\theta(\vn)}\overline{h'_{\theta(\vn+\vk)}} = \delta_{0\eta(\vk)}
\end{equation}
which is the same as Equation \eqref{eq:k1a}.
So $\theta(\vn+\vk)- \theta(\vn)$ is a constant $\eta(\vk)$ independent of $\vn$, in particular
\begin{equation}\label{eq:add}
\eta(\vk) = \theta(\vn+\vk)- \theta(\vn), \forall \vn\in L \subseteq \Lambda_2.
\end{equation}
Let $\vn = \vec{0}$ and notice that $\theta(\vec{0})=0$, we have
\begin{align*}
\eta(\vk) = \theta(\vk), ~ \forall \vk\in  E\cap\Lambda_2^E.
\end{align*}
Hence \eqref{eq:add} becomes
\begin{equation}\label{eq:add2}
\theta(\vk) = \theta(\vn+\vk)- \theta(\vn), ~ \forall \vn\in L\textrm{ and }\forall \vk\in  E\cap\Lambda_2^E.
\end{equation}
Let $\vn = \vk$ in Equation \eqref{eq:add2}, we have
\begin{align*}
\theta(\vk) & = \theta(\vk+\vk)- \theta(\vk)\\
 \theta(2\vk) & =2 ~\theta(\vk).
\end{align*}
And let $\vn = -\vk$ in Equation \eqref{eq:add2}, we have
\begin{align*}
 \theta(-\vk) & =-\theta(\vk)
\end{align*}
By induction, we have
\begin{align*}
 \theta(a\vk) =a ~\theta(\vk), \forall \vk\in E \cap \Lambda_2^E, \forall a\in \Z.
\end{align*}

Let $\{\ve_1,\ve_2\}$ be the standard basis of $\Z^2$.
Let $\vk_1$ be one of $\pm A(2\ve_1)$ that is in $E\cap\Lambda_2^E$ and $\vk_2$ be one of $\pm A(2\ve_2)$ that is in $E\cap\Lambda_2^E$.
Denote $c= \theta(\vk_1),d= \theta(\vk_2)$.
Since $\theta$ is injection, $c$ and $d$ are not zero. We have
\[\theta(c \vk_2) = c \theta(\vk_2) =cd = d\theta(\vk_1)= \theta(d\vk_1). \]
So $\theta$ maps two distinct elements $c \vk_2$ and $d\vk_1$ in $\Lambda_2$ to the same element $cd\in \Z$. A contradiction. Proposition \ref{prop:noiso} is proved.
\end{proof}

Let $A=\left[\begin{array}{cc} 1 & 1 \\ -1 & 1 \end{array}\right]$. Define $\varphi_c$ by
\begin{equation*}
  \hat{\varphi}_c \equiv \frac{1}{2\pi}\chi_{[-\pi,\pi]^2},
\end{equation*}
where the Fourier Transform of $\varphi_c$ is defined as $\displaystyle \hat{\varphi}_c(\vs) \equiv \frac{1}{2\pi} \int_{\R^2} \varphi_c(\vt) e^{-i \vt \cdot \vs} d\mu$.
The function $\varphi_c$ is an orthogonal scaling function associated with matrix $A$ \cite{li} which satisfies the two-scale relation
\begin{align*}
\varphi_c & = \sum_{\vn\in\Z^2} s_{\vn} D_A T_{\vl} ~\varphi_c.
\end{align*}
The set $\{s_\vn\}$ is a solution to the system of equations \eqref{eq:lawton}. Let $\Lambda_c$ denote the support of $\{s_\vn\}$,
$\mathcal{E}_{(\Lambda_c,A,2)}$ denote the reduced system of equations and $\Lambda_c^E$ its index set.
We claim
\begin{description}
  \item[\textsc{Claim}] $L+\vec{e}_1 =  \{\vn =(n,m) \in\Z^2 \mid m+n\in 2\Z-1\} \subseteq \Lambda_c.$
\end{description}

\textit{Proof of Claim}
Assume that $\vl = (m,n) \in \Z^2$ and $m+n$ be odd.
Denote the transpose of $A$ as $A^{\tau}$.
Since $\varphi_c$ is an orthogonal scaling function, we have
\begin{align*}
s_{_{\vl}} & = \langle \varphi_c, D_A T_{\vl} ~\varphi_c   \rangle
     = \langle \hat{\varphi}_c, \hat{D}_A \hat{T}_{\vl} ~\hat{\varphi}_c \rangle
     = \langle \hat{D}_A^{-1}\hat{\varphi}_c, \hat{T}_{\vl} ~\hat{\varphi}_c \rangle
     = \langle D_{A^\tau}\hat{\varphi}_c, e^{-i \vt \cdot \vl} \hat{\varphi}_c \rangle \\
    & = \int_{\R^2} D_{A^\tau} \big( \frac{1}{2\pi} \chi_{[-\pi,\pi]^2}\big) \cdot e^{-i \vt \cdot \vl} \hat{\varphi}_c ~d\mu\\
    & = \frac{1}{\pi\sqrt{2}}\int_{(A^\tau)^{-1}[-\pi,\pi]^2} \cos(mt_1)\cos(nt_2) ~dt_1 dt_2\\
    & =  \left\{\begin{array}{ll}
                \frac{2\sqrt{2}}{\pi} \frac{(-1)^{n+1}}{n^2-m^2}    & \text{ $m,n$ non-zero and $m+n$ odd}\\
                -\frac{4\sqrt{2}}{n^2\pi}     & \text{ $m=0$ and $n$ odd}\\
                -\frac{4\sqrt{2}}{m^2\pi}    & \text{ $n=0$ and $m$ odd}
                \end{array}\right.\\
    & \neq 0.
\end{align*}
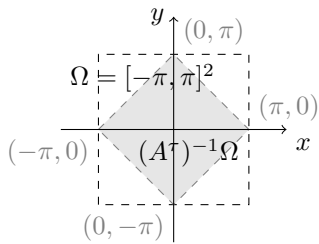
\begin{figure}[h]
\begin{tikzpicture}[scale=1]
\draw[step=.25cm,dashed] (-1,-1) -- (-1,1) -- (1,1) -- (1,-1) -- (-1,-1);
\draw[step=.25cm,gray,dashed,fill=gray9]  (-1,0) node[below left]  {$(-\pi,0)$} -- (0,1)node[above right]  {$(0,\pi)$} --
                                                   (1,0) node[above right]  {$(\pi,0)$} -- (0,-1)node[below left]  {$(0,-\pi)$} -- (-1,0);
\draw[->] (-1.5,0)-- (1.5,0)node[below right] {$x$};
\draw[->] (0,-1.5)-- (0,1.5)node[left] {$y$};
\draw[-] (-1.5,1) node[below right] {$\Omega=[-\pi,\pi]^2$};
\draw[-] (-0.6,0) node[below right] {$(A^{\tau})^{-1}\Omega$};
\end{tikzpicture}
\caption{Support of $(A^\tau)^{-1}\hat{\varphi}_c$}\label{fg_varphi_c}
\end{figure}
The claim is proved.

So $s_{_{\vl}} \neq 0$ for $\vl \in \{\vn =(n,m) \in\Z \mid m+n\in 2\Z-1\} = L+\vec{e}_1$.
Denote the scaling function of $\varphi_c$ after shifting $\ve_1$ as $\varphi_1$.
By Proposition \ref{prop:shift}, $\varphi_1 \simeq \varphi_c$.
$\varphi_1$ is derived from a solution to $\mathcal{E}_{(\Lambda_1,A,2)}$, with support $\Lambda_1 \equiv \Lambda_c-\vec{e}_1$.
Thus $\Lambda_1$ contains the infinite set $L$.
By Proposition \ref{prop:noiso}, the scaling function $\varphi_1$ can not be algebraically isomorphic to a scaling function in $L^2(\R)$.
Hence $\varphi_c$ can not be algebraically isomorphic to a scaling function in $L^2(\R)$.

\section{Acknowledgements}
The authors thank Qing Gu for comments that greatly improved the manuscript.

\end{document}